\documentclass[a4paper,11pt]{article}

\usepackage[T1]{fontenc}      
\usepackage[utf8]{inputenc}   
\usepackage{lmodern}          
\usepackage[english]{babel}   
\usepackage{microtype}        

\usepackage[a4paper, margin=1in]{geometry} 

\usepackage{amsmath, amssymb, amsfonts, amsthm} 
\usepackage{mathtools}        
\usepackage{thmtools}
\usepackage{bm}               
\usepackage{bbm}              
\usepackage{mathrsfs}         
\usepackage{physics}          
\usepackage{tikz-cd}          
\usepackage{enumitem}         

\usepackage{xcolor}           
\usepackage{graphicx}         

\usepackage{hyperref}         
\usepackage{cleveref}         

\usepackage{float}
\usepackage{color,graphicx,stmaryrd}
\definecolor{colorlinks}{RGB}{0, 24, 168}
\definecolor{colorcites}{RGB}{124, 10, 2}
\hypersetup{
	colorlinks=true,
	linkcolor=colorlinks,
	citecolor=colorcites,
	urlcolor=colorlinks,
	pdfborder={0 0 0}
}

\newtheorem{theorem}{Theorem}[section]

\newtheorem{lemma}[theorem]{Lemma}
\newtheorem*{lemma*}{Lemma}
\newtheorem{proposition}[theorem]{Proposition}
\newtheorem{corollary}[theorem]{Corollary}

\theoremstyle{definition}
\newtheorem{definition}[theorem]{Definition}
\newtheorem*{example}{Example}

\theoremstyle{remark}
\newtheorem{remark}[theorem]{Remark}

\newcommand*{\nombres}[1]{\mathbb{#1}}
\newcommand*{\NN}{\nombres{N}}
\newcommand*{\ZZ}{\nombres{Z}}

\newcommand*{\RR}{\nombres{R}}
\newcommand*{\CC}{\nombres{C}}

\newcommand*{\ldef}{\coloneqq}

\newcommand*{\term}[1]{\emph{#1}}

\DeclareMathOperator{\re}{Re} 
\DeclareMathOperator{\im}{Im}

\newcommand{\bigsum}{\mathop{\raisebox{-0.5ex}{\scalebox{1.7}{$\sum$}}}\nolimits}

\begin{document}

\title{A discrete approach to Dirichlet L-functions, their special values and zeros}

\author{Anders Karlsson and Dylan Müller}

\date{May 15, 2026}

\maketitle

\begin{abstract}
We develop a discrete spectral framework for Dirichlet $L$-functions that reveals a combinatorial structure underlying their special values and connects this to their zeros. 

Our approach approximates the classical Dirichlet series by finite spectral sums $L_n(s,\chi)$ associated with cyclic graphs $\mathbb{Z}/n\mathbb{Z}$ and studies their asymptotics as $n\rightarrow \infty$. Combining a refined Euler–Maclaurin expansion with a structural polynomiality property
, we show that at integer arguments  the asymptotic expansions terminate and yield exact identities. This “asymptotic-to-exact” principle produces new infinite families of relations among special values of Dirichlet $L$-functions and recovers, by a different mechanism, formulas previously obtained by Xie–Zhao–Zhao.

An interesting feature of our method is that $\zeta(2n)$ and the corresponding special values for all Dirichlet $L$-functions thereby admit
a finite combinatorial interpretation in terms of rooted spanning forests on any fixed cyclic graph.

Concerning zeros, the same framework leads to some remarks about real zeros and a reformulation of the Generalized Riemann Hypothesis in the case of odd primitive characters in terms of an asymptotic functional equation relating $\xi_n(1-s,\overline{\chi })$ to $\xi_n(s,\chi)$ of the completed discrete functions. 
This establishes the remaining case of the one-dimensional picture obtained in earlier works.

\end{abstract}
\section{Introduction}
Dirichlet $L$-functions, defined by the meromorphic continuation of
\[ L(s,\chi) =\sum_{n \ge 1}\frac{\chi (n)}{n^s}, \]
play a central role in analytic number theory, through their Euler product, special values, functional equation, and zeros. In this paper we use a spectral viewpoint and asymptotics that simultaneously lead to infinite families of new recursive identities among special values (Corollaries \ref{cor:zetaspecial} and \ref{Corollaryspecial}),  exact formulas for certain trigonometric sums (Theorem \ref{Theorem:trig}), and 
a reformulation of the Generalized Riemann Hypothesis in terms of a certain asymptotic functional equation (Theorem \ref{GRH for odd characters}).

Our approach gives the following combinatorial interpretation of Euler's celebrated values $\zeta(2n)$ and the corresponding special values for all Dirichlet $L$-functions. First, they are connected to the analogous discrete spectral functions via asymptotics. Second, the special values of the latter are symmetric rational functions in the eigenvalues and are therefore expressible in terms of the number of spanning forests. Finally, this 
leads eventually to that the classical special values of $\zeta(s)$ and $L(s,\chi)$ can be calculated by simple counting on any single graph $\mathbb{Z}/n\mathbb{Z}$. See Fig.~\ref{fig:spanning} for an illustration in the case of $\zeta(2)$ with $n=2$; the same formula applies for any $n \geq 1$ with its corresponding combinatorial data (the blue numbers). 

\begin{figure}[H] \label{fig:spanning}
\centering
\begin{tikzpicture}[scale=1, >=Latex]

\begin{scope}[shift={(0,0)}]
    \draw[thick] (0,0) circle [radius=0.5];
    \node at (-0.5,0) {$\bullet$};
    \node at (0.5,0) {$\bullet$};
\end{scope}

\draw[->, thick] (1,0.3) -- (4,1.1);

\draw[->, thick] (1,-0.3) -- (4,-1.1);

\begin{scope}[shift={(5,1.1)}]
    \begin{scope}[shift={(0,0)}, scale=0.5]
        \draw[thick] (0,0) arc[start angle=180,end angle=0,radius=1];
        \draw[thick,dotted] (0,0) arc[start angle=180,end angle=360,radius=1];
        \node at (0,0) {\textcolor{red}{$\bullet$}};
        \node at (2,0) {$\bullet$};
    \end{scope}

    \begin{scope}[shift={(2,0)}, scale=0.5]
        \draw[thick] (0,0) arc[start angle=180,end angle=0,radius=1];
        \draw[thick,dotted] (0,0) arc[start angle=180,end angle=360,radius=1];
        \node at (0,0) {$\bullet$};
        \node at (2,0) {\textcolor{red}{$\bullet$}};
    \end{scope}

    \begin{scope}[shift={(4,0)}, scale=0.5]
        \draw[thick,dotted] (0,0) arc[start angle=180,end angle=0,radius=1];
        \draw[thick] (0,0) arc[start angle=180,end angle=360,radius=1];
        \node at (0,0) {\textcolor{red}{$\bullet$}};
        \node at (2,0) {$\bullet$};
    \end{scope}

    \begin{scope}[shift={(6,0)}, scale=0.5]
        \draw[thick,dotted] (0,0) arc[start angle=180,end angle=0,radius=1];
        \draw[thick] (0,0) arc[start angle=180,end angle=360,radius=1];
        \node at (0,0) {$\bullet$};
        \node at (2,0) {\textcolor{red}{$\bullet$}};
    \end{scope}
\end{scope}

\begin{scope}[shift={(5,-1.1)}, scale=0.5]
    \draw[thick,dotted] (1,0) circle [radius=1];
    \node at (0,0) {\textcolor{red}{$\bullet$}};
    \node at (2,0) {\textcolor{red}{$\bullet$}};
\end{scope}

\end{tikzpicture}
 \caption{The graph \( \ZZ/\textcolor{blue}{2}\ZZ \) has $\textcolor{blue}{4}$ rooted spanning trees and $\textcolor{blue}{1}$ rooted spanning $2$-forest. Hence,
 \\
      \protect\parbox{\linewidth}{\centering
        \(\displaystyle \zeta(\textcolor{purple}{2}) = \left( \tfrac {\textcolor{blue}{1}}{\textcolor{blue}{4}}-\tfrac{\textcolor{purple}{2}}{12} \left( -\tfrac{1}{2} \right) \right)\tfrac{(2\pi)^{\textcolor{purple}{2}}}{2\cdot \textcolor{blue}{2}^{\textcolor{purple}{2}}}=\tfrac{\pi^2}{6}.\)
      }
      }

\end{figure}

This asymptotic-to-exact mechanism leads to generalizations of the formulas of Xie-Zhao-Zhao \cite[Theorem A]{XZZ24}, in particular, for any even Dirichlet character $\chi$ mod $q$ and any integers $m,n>0$,
$$
\bigsum\limits_{k=0}^m a_{m-k}(m) \frac{L(2k,\chi)}{(2\pi/qn)^{2k}}=2^{-2m-1}\sum_{j = 1}^{qn-1} \chi(j) \sin^{-2m}(\tfrac{\pi j}{qn}),
$$
where $a_k(s)$ are universal, explicit polynomials of degree $k$ with rational coefficients. 
In the case of the Riemann zeta function $\zeta(s)$, the simplest identity of this type is 
\begin{equation} \label{eqn:zetaspecial}
\bigsum\limits_{k=0}^m a_{m-k}(m)\frac{\zeta(2k)}{(2\pi)^{2k}}=0.
\end{equation}
See Corollary \ref{cor:zetaspecial} for an infinite family of such relations. A version of this corollary already appeared in a paper by Zagier on Verlinde formulas \cite[Theorem 1.iii]{Zag96}, and other types of recursive relations can be found in the works of Euler, Ramanujan and others \cite{SD86, Me17}. 

Our method is to view $\zeta(s)$  and $L(s,\chi)$ as spectral objects arising from the usual Laplacian, and to approximate them by quantities defined in a completely parallel way from the combinatorial Laplacian. This continues the approach set forth in  \cite{CJK10,FK17, F16, MV22, XZZ24}.

An infinite sum, as the Dirichlet series displayed above, is by definition the limit of the partial sums. Here we take instead limits of finite sums with a much richer structure: they are defined entirely analogously as Mellin transforms of heat kernels, like in Riemann's memoir or as emphasized by Jorgenson-Lang \cite{JL11}. Thanks to this structure, these sums have independent interests in various physics and mathematics contexts, see \cite{K20,JKS24} for references. For example, the special value $-\zeta_X '(0)$ is the logarithm of the number of rooted spanning trees, which is an invariant with relevance from electrical circuit theory to the study of water systems \cite{BR18}.

The \emph{spectral zeta function} of a graph or a space $X$ is the Mellin transform of the trace of the heat kernel associated to a Laplacian on $X$: 
\[ \zeta_X(s) \ldef \frac{1}{\Gamma(s)}\int_0^{\infty}\Tr^*(K_X(t))t^s\frac{dt}{t}, \]
where \( \Tr^* \) is the usual trace but discarding the $0$ eigenvalue of the Laplace operator when it is in the point spectrum, and where \( \Gamma \) is the gamma function. For infinite transitive graphs one can use $K_X(t,x,x)$ instead of a trace.

Friedli \cite{F16} studied the following discrete analogue of Dirichlet $L$-functions:
\[ L_n(s,\chi) = \sum_{j = 1}^{qn-1}\frac{\chi(j)}{\left[4\sin^2(\frac{\pi j}{qn})\right]^s}, \]
for even and non-principal \( \chi \) of modulus \( q > 1 \), that is $\chi(-1)=1$. The denominators consist of the non-zero Laplace eigenvalues of the graph of $\mathbb{Z}/qn\mathbb{Z}$. Note that for odd characters, the above function vanishes identically. Therefore, following Xie-Zhao-Zhao in \cite{XZZ24}, we instead define the discrete $L$-function $L_n$ for odd characters of modulus $q>1$ to be
\[
L_n(s,\chi) = \sum_{j=1}^{qn-1} \chi(j) \frac{\cot\left(\frac{\pi j}{qn}\right)}{\left[4\sin^2\left(\frac{\pi j}{qn}\right)\right]^s}.
\]

\subsection{Asymptotics.}
Similarly to \cite{FK17} in the case of the Riemann zeta function, see Theorem~\ref{Asymptotic in the case d=1} below, and by Friedli \cite{F16} for even characters, see Theorem \ref{Asymptotic for even} below, we use 
a modified Euler-Maclaurin formula that allows for endpoint singularities, developed by Sidi in  \cite{Sid04} and \cite{Sid12} in the context of numerical analysis. We get:

\begin{theorem}[Asymptotics for odd characters]\label{Asymptotic for L tilde}
Let $\chi$ be an odd Dirichlet character of modulus $q > 1$. Then, for any fixed $s \in \CC $ and any integer $l \ge 0$, the following asymptotic expansion holds as $n \to \infty$:
\[
L_n(s,\chi) = 2 \sum_{k=0}^{l} b_k(s)  L(1+2(s - k),\chi)\left(\frac{qn}{2\pi}\right)^{1+2(s-k)} + \mathcal{O}(n^{1 +2(s- (l+1))}),
\]
where the functions $b_k(s)$ are polynomials of degree $k$, given by the Laurent expansion
\[
\left(\frac{z/2}{\sin (z/2)}\right)^{2s} \cot (z/2) = \sum_{k \ge 0} b_k(s) z^{2k-1}.
\]
\end{theorem}
The first coefficients are $b_0(s)=2$, $b_1(s)=(s-1)/6$ and $b_2(s) = (5s^2-9s-2)/{720}$.
\subsection{Special values.}
The fact that limits of certain trigonometric sums provide a method to evaluate Euler's special values for $\zeta(s)$ has been observed in several places, for example by Apostol \cite{Ap73} or Witten \cite[p. 177-178]{Wi91}. A remarkable discovery of Xie-Zhao-Zhao \cite{XZZ24} is that there is in fact no need to take a limit. The finite cyclic graph approximates the circle, like in Archimedes, but in contrast, for special values there is no need for any limits. In a precise sense, already the graph on one vertex and one loop knows $\pi^2/6$ (from equation (\ref{eqn:zetaspecial}) above with $m=1$ and  $\zeta(0)=-1/2$). 

Our method is to combine a refined Euler–Maclaurin–Sidi asymptotics with the structural polynomiality of the spectral zeta functions. This gives that, an asymptotic expansion that terminates because of trivial zeros, must in fact coincide with the exact finite $n$ identity. Curiously, this allows one to derive nontrivial closed formulas already at $n=1$ from an argument initially designed for $n\rightarrow\infty$. 

This asymptotic-to-exact method seems applicable in other settings, and we think it will yield further applications.

We prove the following theorem, and refer to \cite{JKS24} for previous results and pointers to the extensive literature about such sums. First, let  \( s \mapsto a_k(s) \) be the polynomial functions of degree \( k \), given by the Taylor expansion 
\[ \left(\frac{z/2}{\sin (z/2)}\right)^{2s} = \sum_{k \ge 0} a_k(s) z^{2k}. \]
The first coefficients are $a_0(s)=1$, $a_1(s)=s/12$, and $a_2(s)=(5s^2+s)/1440$. 

\begin{theorem}\label{Theorem:trig}
For any \( 0 < \theta < 1 \) and integers \( n,m \ge 1 \), the following holds: 
\begin{equation}\label{Eq:Trig:even}
\sum_{j = 0}^{n-1} \sin^{-2m}\left(\frac{\pi(j + \theta)}{n}\right) = 2^{2m}\sum_{k = 0}^m a_{m-k}(m)\zeta_{\RR/\ZZ}(k,\theta)n^{2k},
\end{equation}
and 
\begin{equation}\label{Eq:Trig:odd}
\sum_{j = 0}^{n-1}\cot\left(\frac{\pi(j + \theta)}{n}\right)\left[\sin\left(\frac{\pi(j + \theta)}{n}\right)\right]^{-2m} = -\frac{2^{2m}}{2 \pi m} \sum_{k = 0}^m a_{m-k}(m)\, \partial_\theta \zeta_{\RR/\ZZ}(k,\theta)n^{2k+1}.
\end{equation}
\end{theorem}
The $\zeta_{\RR/\ZZ} (s,\theta)$ is the spectral zeta function associated to a certain bundle over the circle depending on \( \theta \). As such, it is a version of the Hurwitz zeta function, and the corresponding special values at integers involve Bernoulli polynomials and polygamma functions, see section \ref{sec:Hurwitz} below. 
In particular, applying the theorem with $n=1$ recovers part of a well-known reflection relation for the polygamma functions.
The previous theorem can be extended to any integers \( m \in \ZZ \). This is described in section \ref{sec:special}. In particular, for \( -m \le 0 \),
\begin{equation}\label{Eq:special value: theta and negative m}
    \zeta_{\ZZ/n\ZZ}(-m,\theta) = n\zeta_{\ZZ}(-m),
\end{equation} 
for all \( 0 < \theta < 1 \), and all \( n > m \). See the end of this introduction for the definition of $\zeta_{\ZZ}(-m)$ (with no dependence on $\theta$). Here is the case of the Riemann zeta function:
\begin{corollary} \label{cor:zetaspecial}
For any integers \( n,m \ge 1 \), the following holds:
\[ \zeta_{\ZZ/n\ZZ}(m) = \sum_{k = 0}^{m}a_{m-k}(m)\zeta_{\RR/\ZZ}(k)n^{2k}. \]
Explicitly, the case $n=1$ is the formula (\ref{eqn:zetaspecial}) above, and in general for  $n\geq 1$,
\[
\frac{(-1)^{m+1}}{4}\sum_{a=0}^{2m}n^{-a}\left(\begin{array}{c}
2m+1\\
a+1\end{array}\right)
\sum_{j=0}^{a+1}(-1)^{j}\left(\begin{array}{c}
a+1\\
j\end{array}\right)\frac{a+1-2j}{a+1}\left(\begin{array}{c}
m+jn+(a-1)/2\\
2m+a\end{array}\right)
\]
$$
=\bigsum\limits_{k=0}^m a_{m-k}(m)\frac{\zeta(2k)}{(2\pi)^{2k}}n^{2k}.
$$
\end{corollary}

The left hand side in the last formula comes from \cite{CM99,BY02}.
The formulas in Corollary \ref{cor:zetaspecial} should be compared with those in Zagier's \cite{Zag96}. It seems unlikely that he thought of the trigonometric sums as special values of graph zeta functions.  
See also  \cite[Theorem B]{XZZ24}. 
There are also more isolated examples of such recursive relations in the literature, see \cite{SD86, Me17} and references therein.

Furthermore, the theorem has the following corollary:

\begin{corollary}\label{Corollaryspecial}
Let \( \chi \) be an even Dirichlet character of modulus \( q \) (not necessarily primitive). Then for any integers \( n \ge 1 \) and \( m \ge 0 \), 
\[ L_n(m,\chi) = 2 \sum_{k = 0}^{m}a_{m-k}(m)L(2k,\chi) \left(\frac{qn}{2\pi}\right)^{2k}. \]
Let \( \chi \) be an odd Dirichlet character of modulus \( q \) (not necessarily primitive). Then for any integers \( n \ge 1 \) and \( m \ge 0 \), 
\[L_n(m,\chi) = 2\sum_{k = 0}^{m}b_{m-k}(m)\,L(2k+1,\chi) \left(\frac{qn}{2\pi}\right)^{2k+1}. \]
\end{corollary}

In the case $n=1$ these formulas were found by Xie, Zhao, and Zhao in \cite{XZZ24}. They also displayed the corresponding formulas expressing the special values of the Dirichlet $L$-functions in terms of the discrete ones.

Special values at negative integers \( -m < 0 \) of discrete \( L \) functions satisfy 
\[ L_n(-m,\chi) = 0, \]
whenever \( n > m \) (see the proof of Corollary \ref{Polynomial behaviour}). These values should not be dismissed as trivial. For example, when \( \chi \) is real, even and non-principal the condition \( L_n(-m,\chi) \ge 0 \) for all even \( n \ge 2 \) and for any integers \( m > 0 \) implies the absence of real zeros in the critical strip of the corresponding Dirichlet \( L \)-function.

The case \( m = 0 \) for odd characters in Corollary \ref{Corollaryspecial} has a special interest from the Dirichlet class number formula and reads
\[ L(1,\chi) = \frac{\pi}{2}\frac{1}{qn}\sum_{j = 1}^{qn-1}\chi(j)\cot\left(\frac{\pi j}{qn}\right), \]
for all \( n \ge 1 \). By trigonometric manipulations, this formula reduces to the classical cotangent formula (the case \( n=1 \)) for \( L(1,\chi) \). 

\subsection{Zeros of Dirichlet $L$-functions.}
The asymptotics in Theorem~\ref{Asymptotic for L tilde} that was used to derive identities among special values, provide, at the same time, information about zeros. We begin by:

\begin{corollary} \label{Corollary about Siegel zeros for odd Dirichlet characters}
Let $\chi$ be a real, odd Dirichlet character. If, for each real $0 < s < 1$, there is a subsequence of $n\rightarrow \infty$ such that $L_n(s,\chi)\geq 0$, then the associated Dirichlet $L$-function $L(s,\chi)$ has no Siegel zero.
\end{corollary}

Note that in the asymptotics above, for $s$ in the critical strip, the values of the Dirichlet $L$-function involved are all outside the critical strip. So in some sense, the existence of Siegel zeros is linked to certain values of $L$ outside of the critical strip.

Not only is there a connection via the sign, or the vanishing, of $L_n(s,\chi)$, to zeros of $L(s,\chi)$, but there is a more surprising connection. Indeed, the Riemann Hypothesis (RH) in question is equivalent to graph zeta- or $L$-function asymptotic functional equations. It has been established in the following cases:
\begin{itemize}
    \item $\zeta(s)$ in \cite{FK17},
    \item $L(s,\chi)$ for primitive even Dirichlet character $\chi$ in \cite{F16}
    \item the Dedekind zeta function of the field of Gaussian rationals in \cite{MV22} (requiring a $9$-point Laplacian).
\end{itemize}

In each case, the form and algebra of the proofs differ. In all cases, it relies on nontrivial monotonicity results of Matiyasevich-Saidak-Zvengrowski \cite{MSZ14}. To this list we can now add the odd character version, which completes the picture in the $1$-dimensional case. Let
\[
\xi_n(s,\chi) = 2^sn^{-s} \left( \frac{\pi}{q} \right)^{\frac{s}{2}} \Gamma\left(\frac{s+1}{2}\right) L_n\left(\frac{s-1}{2},\chi\right).
\] 
Note that apart from the $2^s n^{-s}$, these fudge factors are what is needed to complete the corresponding Dirichlet $L$-function to have its functional equation in the symmetric form:
\[ \xi(s,\chi) = w_\chi\xi(1-s,\overline{\chi}). \]
The factor $w_{\chi}$ involves a Gauss sum, and has modulus $1$ if $\chi$ is primitive, see section \ref{sec:GRH} below.
We obtain:

\begin{theorem}\label{GRH for odd characters}
Let $\chi$ be a primitive, odd Dirichlet character of conductor $q > 1$. The following two statements are equivalent:
\begin{enumerate}[label=(\roman*)]
\item All zeros of the Dirichlet $L$-function $L(s,\chi)$ in the strip $0 < \Re s < 1$ lie on the critical line $\Re s = \tfrac{1}{2}$.
\item For all $s$ in the critical strip $0 < \Re s < 1$,
\[
\lim_{n \to \infty} \left| \frac{\xi_n(s,\chi)}{\xi_n(1-s,\overline{\chi})} \right| = 1.
\]

\end{enumerate}
\end{theorem}

In view of the standard factorization of the above-mentioned Dedekind zeta function, there are thus two graph reformulations of its RH: via cyclic graphs and via two-dimensional discrete tori.

Possible real zeros in the critical strip is a very important issue with numerous number theoretic consequences. Here is a variant of the GRH-reformulation restricted to real zeros:

\begin{corollary}\label{GRH in the real line for odd characters}
Let $\chi$ be a primitive, odd, real Dirichlet character of conductor $q > 1$. Take $0< s <1$, $s \neq 1/2$. Then $L(s,\chi)\neq 0$ is equivalent to
\[
\lim_{n \to \infty} \frac{\xi_n(s,\chi)}{\xi_n(1-s,\overline{\chi})} = 1.
\]
\end{corollary}

We tried to use the discrete approach to rule out real zeros, but could only establish:
\begin{theorem}\label{thm:positivemean2}
    Let \( \chi \) be a Dirichlet character, primitive of conductor \( q > 1 \), real, even, and with the positive mean property. Then $\chi $ is heat positive, $\tau(\chi)=\sqrt q$, and $L(s,\chi)>0$ for $0<s<1$.
\end{theorem}

See section \ref{sec:realseros} for terminology. For these characters, the positivity of the Dirichlet $L$-function is already well-known from Chowla's clever and simple proof in \cite{Ch36} (corresponding to the case $m(\chi)=2$). We still stick to the notion that the discrete viewpoint, with one or another approach, could lead to new results on real zeros.

\subsection{Further context.}
From one particular spectral viewpoint, Riemann's zeta function $\zeta(s)$ and Dirichlet's $L$-functions are associated with the circle $\mathbb{R}/\mathbb{Z}$. Selberg's zeta function \cite{Se56} is similarly associated with surfaces $\mathbb{H}^2/\Gamma$. Here we study the zero-dimensional case $\mathbb{Z}/n\mathbb{Z}$. Ideas and speculations on how zeta functions in various dimensions interact can be found in \cite{JL11}.

It should be said that the closest graph analogue of Selberg's zeta function is instead the Ihara zeta functions, for example since they both are Laplace transforms of heat kernels, see \cite{CJK15}. The Ihara zeta function for general graphs was suggested by Serre in the 1970s \cite[p. IX]{S02} after reading Ihara's work. It was developed by Sunada and others into an important chapter of spectral graph theory \cite{Te10}. These zetas (and the $L$-functions of Stark-Terras \cite{ST96}) have analogues with number theoretic zeta functions. In contrast, the functions associated to graph spectra discussed here have more direct connections. Selberg's zeta function, with its Euler product and trace formula, has features analogous to classical number theory. But in addition, for the modular surface, the Riemann zeta function actually appears in the functional equation. 

 One fundamental graph spectral zeta function is the one associated to the discrete line $\mathbb{Z}$:
 $$
 \zeta_{\mathbb Z}(s) = \frac{1}{\Gamma(s)}\int_0^{\infty} e^{-2t}I_0(2t)t^s \frac{dt}{t}=\frac{1}{2^{2s}\pi^{1/2}} \frac{\Gamma(1/2-s)}{\Gamma(1-s)} =\binom{-2s}{-s}
 =\prod_{k=1}^{\infty}\frac{(1-s/k)^2}{(1-2s/k)},
 $$
see \cite{FK17,KP23}. It has its share of analogues, for example a functional equation of the usual form $\xi_{\mathbb Z}(1-s)=\xi_{\mathbb Z}(s)$. Moreover, special values of $\zeta_{\mathbb Z}$ can be used to give the formula of the volumes of the round spheres in a form similar to the Minkowski-Siegel volume formulas for spaces like $\mathrm{SL}_n(\mathbb{R}) /\mathrm{SL}_n(\mathbb{Z})$ \cite{KP23}. But $\zeta_{\mathbb Z}(s)$ also actually appears in the functional equation of Eisenstein series \cite[(3.7),(3.9)]{Se56}, and of Selberg's zeta function \cite[p. 79]{Se56} and \cite[§10]{Se54}. Thus, in the case of the modular surface, both $\zeta(s)$ and $\zeta_{\mathbb Z}(s)$ appear in the fudge factors for the Selberg zeta function (\emph{cf.} \cite[p. 300]{JL11}). This may not be surprising since both $\mathbb{R}/ \mathbb{Z}$ and $\mathbb{Z}$ can be thought related to the cusp. 
 The graph zeta function $\zeta_{\mathbb Z}(s)$ is also present in the Gindikin-Karpelevich formulas \cite{K20} and in the significant scaling-limit context of \cite{ABS25}.
 
 Moreover, as Sarnak pointed out to us, it is suggestive that the special values at $s=-m$ are exactly the integers exploited by Chebyshev in his work on primes. Thus the graph special values,  $\zeta_{\mathbb Z}(-m)$ and $L_n(-m,\chi) $, all have implications for the distribution of primes.

\paragraph*{Acknowledgments.
}
This work was supported by the Swiss NSF Grants 200020-200400, 200021-212864 and the Swedish Research Council Grant 104651320. The authors would like to thank the Isaac Newton Institute for Mathematical Sciences, Cambridge, for support and hospitality during the program Operators, Graphs, Groups, where work on this paper was undertaken. This work was supported by EPSRC grant no EP/Z000580/1.


\section{Bundle Laplacian and asymptotics}
In this section we define a bundle Laplacian on the discrete circle \( \ZZ/n\ZZ \) and its continuous counterpart \( \RR/\ZZ \). We then establish an asymptotic expansion for the spectral zeta function involved as the size of the discrete circle \( n \to \infty \).

\subsection{Hurwitz zeta functions} \label{sec:Hurwitz}

The \term{Hurwitz zeta function} is defined by the series 
\[ \zeta(s,\theta) = \sum_{n \ge 0}(n+\theta)^{-s} \]
for \( 0 < \theta < 1 \) and \( \Re(s) > 1 \). It can be analytically continued to the whole complex plane, with a simple pole at \( s = 1 \) of residue \( 1 \). The special values of the Hurwitz zeta function at negative integers satisfy 
\[ \zeta(-n,\theta) = -\frac{B_{n+1}(\theta)}{n+1}, \]
for \( n \ge 0 \), where \( B_{n}(x) \) are the Bernoulli polynomials defined by the generating series 
\[ \frac{te^{xt}}{e^t-1} = \sum_{n \ge 0}B_n(x) \frac{t^n}{n!}. \]
From this, we observe that 
\begin{equation}\label{Bernoulli relations}
B_n(1-x) = (-1)^nB_n(x).
\end{equation} 
For positive integers \( n \), the formulas are known to be related to the polygamma functions:
\[
\zeta(n,\theta)=\frac{(-1)^n}{(n-1)!} \psi^{(n-1)}(\theta),
\]
where the digamma function $\psi(s)$ is defined by \( \psi(z)=\Gamma '(z)/\Gamma(z) \).

\subsection{Bundle Laplacian}
The notion of a \term{vector bundle} over a finite graph is well explained in \cite{Ke11}. To a finite graph \( G \) we can attach to each vertex a copy of a given vector space \( V \) together with the data of a \term{connection}, that is a map \( \phi \colon E_G \rightarrow \mathrm{GL}(V) \), where \( E_G \) is the set of edge of \( G \) and with the condition that \( \phi(\overline{e}) = (\phi(e))^{-1} \), where \( \overline{e} \) is the opposite edge of \( e \). The bundle is said to be \term{unitary} if the connection takes values in \( U(V) \). If \( V \) is one dimensional, we say that it is a \term{line bundle}. If \( \gamma \colon \{1,\ldots,n\} \rightarrow E_G \) is a closed path in \( G \), the \term{monodromy} around \( \gamma \) is 
\[ m_\phi(\gamma) \ldef \phi(\gamma(n)) \cdot \ldots \cdot\phi(\gamma(1)) \in \mathrm{GL}(V). \]
To each \( V \)-bundle over \( G \), we may associate a \term{bundle Laplacian} defined by 
\[ (\Delta_{\phi}f)(x) = \sum_{y \sim x}(f(x)-\phi_{(x,y)}f(y)) \]
on functions \( f \colon G \rightarrow V \).

\subsubsection{A unitary line bundle over the discrete circle}\label{subsec:DiscreteBundle}
Let \( C_n \) be the Cayley graph of \( \ZZ/n\ZZ \) with generator \( \{ 1 \} \). For each \( \theta \in \RR \), we define a unitary line bundle over \( C_n \) associated with the constant connection \( \phi_{\theta}((x,x+1)) = e^{2\pi i \theta/n } \). It has monodromy \( m_\theta(C_n) = e^{2\pi i \theta} \). The bundle Laplacian is therefore equal to 
\[ (\Delta_{n,\theta}f)(x) = 2f(x)-e^{2 \pi i \theta/n}f(x+1) - e^{-2\pi i \theta/n}f(x-1), \]
for functions \( f \colon C_n \rightarrow \CC \). In fact, \( \Delta_{n,\theta} = (d_{\theta})^*d_{\theta} \) for the difference operator \( d_{\theta} \colon l^2(C_n) \rightarrow l^2(E_{C_n}) \) defined by 
\[ (d_{\theta}f)(e) = e^{\pi i \theta/n}f(e^+) - e^{-\pi i \theta/n}f(e^-), \]
where \( e = (e^-,e^+) \) is the edge that starts from \( e^- \) and ends at \( e^+ \). It is diagonalized by Fourier with spectrum 
\[ \sigma(\Delta_{n,\theta}) = \left\{4\sin^2\left(\frac{\pi(j+\theta)}{n} \right) \mid j = 0,\ldots,n-1 \right\}. \]
When \( \theta \notin \ZZ \) the spectral zeta function associated with \( \Delta_{n,\theta} \) is the entire function
\begin{equation}
\zeta_{\ZZ/n\ZZ}(s,\theta) = \zeta_n(s,\theta)  \ldef \sum_{j = 0}^{n-1}\left[4\sin^2\left(\frac{\pi(j+\theta)}{n}\right)\right]^{-s}.
\end{equation}
If \( \theta \in \ZZ \) the bundle Laplacian is unitary equivalent to the standard combinatorial Laplacian. Therefore its zeta function is the spectral zeta function on \( \ZZ/n\ZZ \), 
\[ \zeta_{\ZZ/n\ZZ}(s) = \sum_{j = 1}^{n-1}\left[4\sin^2\left(\frac{\pi j}{n}\right)\right]^{-s}. \]
We will also need the derivative with respect to \( \theta \) of this spectral function: 
\begin{equation}
\partial_{\theta}\zeta_n(s,\theta) = -s\frac{2\pi}{n} \sum_{j = 0}^{n-1}\cot\left(\frac{\pi(j + \theta)}{n}\right)\left[4\sin^2\left(\frac{\pi(j + \theta)}{n}\right)\right]^{-s}.
\end{equation}

\subsubsection{The continuous circle}
In the same way, we define the bundle Laplacian on the circle associated with the unitary connection \( \mathrm{d} + 2\pi i\theta \mathrm{d}x \) where we identify the circle with \( \RR/\ZZ \). We denote it by \( \Delta_{\theta} \) and it is simply 
\[ \Delta_{\theta} = -(\mathrm{d}+2\pi i \theta)^2. \]
It is again diagonalized by Fourier and has spectrum 
\[ \sigma(\Delta_{\theta}) = \{4\pi^2(n+\theta)^2 \mid n \in \ZZ \}. \] 
For \( \theta \notin \ZZ \), the spectral zeta function associated to \( \Delta_{\theta} \) is therefore 
\[ \zeta_{\RR/\ZZ}(s,\theta) = \frac{1}{(2\pi)^{2s}}\sum_{n \in \ZZ} |n+\theta|^{-2s} = (2\pi)^{-2s}\left[\zeta(2s,\theta) + \zeta(2s,1-\theta)\right]. \] 
The relation \eqref{Bernoulli relations} implies that \( \zeta_{\RR/\ZZ}(s,\theta) \) vanishes at negative integers \( n \le 0 \). 
If \( \theta \in \ZZ \), again \( \Delta_{\theta} \) is unitary equivalent to the usual Laplacian, its spectral zeta is therefore the spectral zeta function of the circle
\[ \zeta_{\RR/\ZZ}(s) = 2(2\pi)^{-2s}\zeta(2s), \]
where \( \zeta(s) \) is the Riemann zeta function. Observe that 
\begin{equation}
\zeta_{\RR/\ZZ}(s) = \lim_{\theta \to 0} \left( \zeta_{\RR/\ZZ}(s,\theta)-(2\pi \theta)^{-2s} \right).
\end{equation}
Analogously to the discrete case, we compute its derivative with respect to \( \theta \);
\[ \partial_{\theta}\zeta_{\RR/\ZZ}(s,\theta) = -2s(2\pi)^{-2s}\bigl[\zeta(1+2s,\theta) - \zeta(1+2s,1-\theta)\bigr]. \]
Observe that for \( 0 < \theta < 1 \), this is an entire function of \( s \), with 
\[-\frac{1}{2s}\partial_{\theta}\zeta_{\RR/\ZZ}(s,\theta) = \psi(1-\theta)-\psi(\theta) + \mathcal{O}(s) = \pi\cot\pi\theta + \mathcal{O}(s). \]
\subsection{Asymptotics}
The main tool of this article is the following asymptotic expansion deduced from the Euler-MacLaurin formula adapted for singularities by Sidi in \cite{Sid12}.
\begin{theorem}\label{Asymptotic for the bundle Laplacian}
For any \( s \neq \tfrac 12, \tfrac 32, \tfrac 52 ,\ldots \), any \( 0 < \theta < 1 \), and any \( l \ge 0 \), the following asymptotics hold as \( n \to \infty \):
\begin{equation}
\zeta_{\ZZ/n\ZZ}(s,\theta)  =  n\zeta_\ZZ(s) + n^{2s}\left(\sum_{k = 0}^{l} a_k(s) \zeta_{\RR/\ZZ}(s-k,\theta)n^{-2k} + \mathcal{O}(n^{-2-2l})\right),
\end{equation} 
and without restriction on \( s \),
\begin{equation}
\partial_{\theta}\zeta_{\ZZ/n\ZZ}(s,\theta)  = n^{2s}\left(\sum_{k = 0}^{l} a_k(s) \partial_{\theta}\zeta_{\RR/\ZZ}(s-k,\theta)n^{-2k} + \mathcal{O}(n^{-2-2l})\right),
\end{equation}
where the coefficients \( a_k(s) \) are defined in the introduction.
\end{theorem}
\begin{remark}
    One might be tempted to deduce the second formula by differentiating the first with respect to \( \theta \). We believe that this argument can be made rigorous, although we do not have control on the rest present in the asymptotic with respect to \( \theta \). We therefore instead apply the Euler-MacLaurin-Sidi formula twice.
\end{remark}

\begin{proof}[Proof of Theorem \ref{Asymptotic for the bundle Laplacian}]
It is a direct application of Theorem 2.1 of \cite{Sid12}. Take any \( 0 < \theta < 1 \), and \( s \neq \tfrac 12, \tfrac 32, \tfrac 52, \ldots \) and consider the function \( f_s(z) = [4\sin^2(\pi z)]^{-s} \), which has the asymptotic for \( z \to 0^+ \), and by symmetry, as well for \( z \to 1^-\) :
\[ f_s(z) = \sum_{k \ge 0}a_k(s)(2\pi z)^{2(k-s)}. \]
Theorem 2.1 of \cite{Sid12} implies that 
\begin{align*}
    \sum_{j = 0}^{n-1}f_s\left(\frac{j+\theta}{n}\right) & = n\int_0^1f_s(x)\mathrm{d}x + \sum_{k = 0}^{l}a_k(s)(\zeta(2(s-k),\theta) + \zeta(2(s-k),1-\theta))\left(\frac{2\pi}{n}\right)^{2(k-s)} \\
    & + \mathcal{O}(n^{2(s-l-1)}).
\end{align*}
The first asymptotics follow by recalling that
\[ \zeta_{\ZZ}(s) = \int_0^1\left[4\sin^2(\pi x)\right]^{-s}\mathrm{d}x, \]
and rearranging the terms.
To obtain the second asymptotics, we used again Theorem 2.1 in \cite{Sid12}, but with the function \( g_s(z) \ldef \cot(\pi z)f_s(z) \), which is odd around the line \( \re z = 1/2, \) that is, \( g_s(1-z) = -g_s(z) \), and therefore has asymptotic
\[ g_s(z) = \sum_{k \ge 0}b_k(s)(2\pi z)^{2(k-s)-1}, \]
for \( z \to 0^+ \) and opposite asymptotic for \( z \to 1^- \). By the Euler-MacLaurin-Sidi formula we obtain for any \( s \in \CC \)
\begin{align*}
    \sum_{j = 0}^{n-1}g_s\left(\frac{j+\theta}{n}\right) & = n\int_0^1g_s(x)\mathrm{d}x +  \frac{n}{2 \pi}\delta_s(\psi(1-\theta)-\psi(\theta))\\
    & + \underset{k \neq s}{\sum_{0 \le k \le l}} b_k(s)(\zeta(1+2(s-k),\theta)-\zeta(1+2(s-k),1-\theta))\left(\frac{2\pi}{n}\right)^{2(k-s)-1} \\
    & + \mathcal{O}(n^{1+2(s-l-1)}),  
\end{align*} 
where \( \delta_s = 0 \) if \( s \notin \ZZ_{\ge 0} \), and \( \delta_s = b_s(s) \) otherwise. Now observe first that since \( g_s \) is odd around \( \re z = 1/2 \), the integral is identically \( 0 \). We can incorporate the term with \( \delta_s \) in the main sum and rewrite the asymptotic as 
\begin{align*}
    \sum_{j = 0}^{n-1}g_s\left(\frac{j+\theta}{n}\right) & = - \frac{n}{2\pi s}\Bigg [\sum_{k = 0}^{l}b_k(s)\frac{s}{2(s-k)}\partial_{\theta}\zeta_{\RR/\ZZ}(s-k,\theta)n^{2(s-k)} \\
    & + \mathcal{O}(n^{2(s-l-1)})\Bigg], 
\end{align*}
which is the desired asymptotic when multiplied by \( -\tfrac{2\pi s}{n} \), since 
\[ a_k(s) = b_k(s) \frac{s}{2(s-k)}. \]
\end{proof}

\section{Asymptotics of discrete Dirichlet $L$-functions}
\label{Section for asymptotic for odd character}
In this section we prove Theorem \ref{Asymptotic for L tilde}.
For comparison, let us recall the previous results slightly reformulated:

\begin{theorem}[Theorem 1.1 in \cite{BHS08}; Theorem 0.3 in \cite{FK17}]
\label{Asymptotic in the case d=1}
For any \( s \neq \frac 12, \frac 32, \frac 52 ,\ldots \) and any \( m \ge 0 \), the following asymptotic holds, as \( n \to \infty \),
\begin{align*}
\zeta_{\ZZ/n\ZZ}(s) & =  n\zeta_{\mathbb{Z}}(s)  + n^{2s}\left(\sum_{k = 0}^{l} a_k(s) \zeta_{\mathbb{R}/\mathbb{Z}}(s-k)n^{-2k} + \mathcal{O}(n^{-2-2l})\right). 
\end{align*} 
\end{theorem}

Here \( \zeta_{\mathbb{R}/\mathbb{Z}}(s) = 2(2\pi)^{-2s}\zeta(2s) \) is the spectral zeta function of the circle involving the Riemann zeta function. Moreover, the \( s \mapsto a_k(s) \) are polynomials functions of degree \( k \) and that were defined in the introduction. In fact, this theorem is the \( \theta = 0 \) version of Theorem \ref{Asymptotic for the bundle Laplacian}. However, since we do not have any control of the rest with respect to \( \theta \), one needs again to use the modified Euler-MacLaurin formula of \cite{Sid12} to obtain it. We will not bother reproving it here.

The following asymptotic expression is obtained by Friedli in \cite{F16}.
\begin{theorem}[\cite{F16}]\label{Asymptotic for even}
Let \( \chi \) be an even, non-principal Dirichlet character of modulus \( q > 1 \). Then, for any fixed \( s \in \CC \setminus \{\frac12,\frac32,\frac52,\ldots\} \) and any integer \( m \ge 0 \), the following asymptotic expansion holds as \( n \to \infty \):
\[ L_n(s,\chi) = 2 \left(\frac{2\pi}{qn}\right)^{-2s}\left[\sum_{k = 0}^{l}a_k(s) \left(\frac{2\pi}{qn}\right)^{2k}L(2(s-k),\chi) + \mathcal{O}(n^{-2-2l})\right], \]
where the \( a_k(s) \) are defined in the introduction.
\end{theorem}  

We now proceed to prove this theorem and its odd counterpart, Theorem \ref{Asymptotic for L tilde} as corollaries of Theorem \ref{Asymptotic for the bundle Laplacian}.

\begin{proof}[Proof of Theorem \ref{Asymptotic for even} and Theorem \ref{Asymptotic for L tilde}]
For any non-principal Dirichlet character \( \chi \) of modulus \( q > 1 \), let \( a = 0 \) if \( \chi \) is even and \( a = 1 \) if \( \chi \) is odd. Then we have 
\[ \sum_{r=1}^{q-1}\chi(r)(\partial_\theta^a\zeta_n)(s,r/q)= \left(-s\frac{2\pi}{n}\right)^a \cdot L_n(s,\chi). \]
and 
\[ \sum_{r = 1}^{q-1}\chi(r)(\partial_\theta^a\zeta_{\RR/\ZZ})(s,r/q) =(-4s\pi)^a \cdot 2 \left(\frac{q}{2\pi}\right)^{2s+a}L(2s+a,\chi). \]
Therefore, for any \( \chi \) non-principal, we obtain :
\begin{align*}
    L_n(s,\chi) & = \sum_{r = 1}^{q-1}\chi(r)\partial_\theta^a\zeta_n(s,r/q)\left(\frac{n}{-2\pi s}\right)^a \\
    & = n^{2s}\left(\sum_{k = 0}^{l}a_k(s)\left(\frac{1}{-2\pi s}\right)^a\sum_{r = 1}^{q-1}\chi(r)(\partial_\theta^a\zeta_{\RR/\ZZ})(r/q,s-k)n^{a-2k} + \mathcal{O}(n^{a-2-2l}) \right) \\
    & = 2n^{2s}\left(\sum_{k = 0}^{l}a_k(s)\left(\frac{(2s-k)}{s}\right)^a \left(\frac{q}{2\pi}\right)^{2(s-k)+a}L(a+2(s-k),\chi)n^{a-2k} + \mathcal{O}(n^{a-2-2l})\right) \\
    & = 2 \left(\frac{qn}{2\pi}\right)^{2s}\left(\sum_{k=0}^{l}a_k(s)\left(\frac{2(s-k)}{s}\right)^aL(a+2(s-k),\chi)\left(\frac{qn}{2\pi}\right)^{a-2k} + \mathcal{O}(n^{a-2-2l})\right).
\end{align*}
This concludes the proof of both Theorem \ref{Asymptotic for L tilde} and Theorem \ref{Asymptotic for even}, recalling that 
\[ a_k(s) \frac{2(s-k)}{s} = b_k(s). \]
\end{proof}
\begin{remark}
    The restriction on \( s \) in Theorem \ref{Asymptotic for even} can be omitted. Indeed, applying the Euler-MacLaurin-Sidi when \( s \) is a half integer produces \( \log \) terms that vanish when summed against the character. But in order not to lengthen the present paper, we did not include this argument here.
\end{remark}

In fact, in the next section, what will be most important for us is the two first terms of the asymptotic, that is the case \( l = 1 \) in Theorem \ref{Asymptotic for L tilde}. Precisely the asymptotic of size \( l = 1 \) is
\begin{equation}\label{Asymptotic of size 2 for L tilde}
L_n(s,\chi) = 2\left(\frac{qn}{2 \pi}\right)^{1+2s}\left[2L(1+2s,\chi) + \frac{s-1}{6}L(2s-1,\chi)\left(\frac{\pi}{qn}\right)^{2} + \mathcal{O}(n^{-4}) \right],
\end{equation}
because \( b_0(s) = 2 \), and \( b_1(s) = \tfrac{s-1}{6} \).

\paragraph*{Proof of the corollary about Siegel zeros.}

Let $\chi$ be a non-principal, even and real Dirichlet character of modulus $q>1$. Fix $0<s<1$. Then $L(s-2,\chi) <0$ and it follows from the asymptotics in Theorem \ref{Asymptotic for even} that if $L_n(s,\chi)\geq 0$ for a subsequence of $n\rightarrow \infty$, then $L(s,\chi)>0$ and in particular non-zero.

Let $\chi$ be an odd and real Dirichlet character of modulus $q$.  Fix $0<s<1$. The second term in the asymptotics in Theorem \ref{Asymptotic for L tilde} now has a different sign, since $(s-3)/6<0$. It follows that if $L_n(s,\chi)\geq 0$ for a subsequence of $n\rightarrow \infty$, then $L(s,\chi) \geq 0$. 

Since Siegel zeros are known to be simple for real characters, they cannot exist if $L(s,\chi) \geq 0$ for all $0<s<1$ (or actually, it is enough with a certain interval $1-\epsilon <s<1$.)

This discussion establishes Corollary \ref{Corollary about Siegel zeros for odd Dirichlet characters}.

\section{Equivalence with GRH for odd character.} \label{sec:GRH}
In this section, we prove Theorem \ref{GRH for odd characters}.
Recall that the following result is proven in \cite{F16}.
\begin{theorem}[Theorem 1.3 of \cite{F16}]\label{GRH equivalence for even}
Let \( \chi \) be a primitive and even character of conductor \( q > 1 \). The completed \( L \)-function of \( \ZZ/qn \ZZ \) is defined by
\[ \xi_n(s,\chi) \ldef 2^s\,n^{-s}\left(\frac{\pi}{q}\right)^{\frac{s}{2}}\Gamma\left(\frac{s}{2}\right)L_n(\tfrac s2,\chi), \]
for \( 0 < \re s < 1 \). Then the following statements are equivalent 
\begin{enumerate}[label = (\roman*)]
\item For all \( 0 < \re s < 1 \) such that \( \im s \ge 8 \), we have 
\[ \lim_{n \to \infty}\left|\frac{\xi_n(s,\chi)}{\xi_n(1-s,\overline{\chi})}\right| = 1. \]
\item In the region \( \im s \ge 8 \), all the zeros of the \( L \)-function associated to \( \chi \) have real part \( 1/2 \).
\end{enumerate}
\end{theorem}
For \( \chi \) odd, primitive and of conductor \( q > 1 \), the completed Dirichlet \( L \)-function is defined to be
\[ \xi(s,\chi) \ldef \left(\frac{\pi}{q}\right)^{-\frac{s}{2}}\Gamma\left(\frac{s+1}{2}\right) L(s,\chi), \]
and it satisfies the functional equation 
\[ \xi(s,\chi) = w_\chi\xi(1-s,\overline{\chi}), \]
where
\[ w_\chi = \frac{\tau(\chi)}{i\sqrt{q}}, \]
with \( \tau(\chi) \) the Gauss sum 
\[ \tau(\chi) \ldef \sum_{j = 1}^{q} \chi(j)e^{\frac{2 \pi i j}{q}}. \]
It satisfies \( |\tau(\chi)| = \sqrt{q} \) and therefore \( |w_\chi| = 1 \).

The proof of Theorem \ref{GRH for odd characters} that we present here is an adaptation of the proof of Theorem \ref{GRH equivalence for even} presented in \cite{F16}, with some modifications coming from \( \chi \) being odd. We start with a lemma, which is proven in \cite{MSZ14}
\begin{lemma}[Corollary 2.6L in \cite{MSZ14}]
Fix \( t \) real and write \( s = \sigma + it \), then the function 
\[ \sigma \mapsto |\xi(s-2,\chi)|, \]
is strictly decreasing in the interval \( 0 < \sigma < 1 \).
\end{lemma}
In fact, in Corollary 2.6L of \cite{MSZ14}, they showed that
\[ \sigma \mapsto |\xi(s,\chi)| \] 
is strictly decreasing in the region \( \sigma \le 0  \) for any non-principal primitive character.
\begin{proof}[Proof of Theorem \ref{GRH for odd characters}]
Recall that 
\[ \xi_n(s,\chi) \ldef 2^sn^{-s}\left(\frac{\pi}{q}\right)^{\frac{s}{2}}\Gamma\left(\frac{s+1}{2}\right)L_n\left(\frac{s-1}{2},\chi\right), \]
and therefore the asymptotic \eqref{Asymptotic of size 2 for L tilde} gives 
\[ \xi_n(s,\chi) = 4\xi(s,\chi) + \beta(s,\chi)n^{-2} + \mathcal{O}(n^{-4}) \]
where 
\[ \beta(s,\chi) = \frac{\pi}{q}\frac{(s-1)(s-3)}{6} \xi(s-2,\chi). \]
First observe that if \( \xi(s,\chi) \neq 0 \) then clearly we have
\[ \lim_{n \to \infty} \frac{\xi_n(s,\chi)}{\xi_n(1-s,\overline{\chi})} = w_\chi. \]
However, since \( 0 < \re s < 1 \), we know that \( \beta(s,\chi) \neq 0 \) and \( \beta(1-s,\overline{\chi}) \neq 0 \), therefore if \( \xi(s,\chi) = 0 \),  we have that the following equality
\[ \lim_{n \to \infty} \left|\frac{\xi_n(s,\chi)}{\xi_n(1-s,\overline{\chi})}\right| = 1 \]
is equivalent to 
\[ |\beta(s,\chi)| = |\beta(1-s,\overline{\chi})|. \]
Now, the previous lemma implies directly that \( |\beta(s,\chi)| \) is strictly decreasing in the strip \( 0 < \re s < 1 \), for any non-principal primitive character \( \chi \). This shows that the equality \( |\beta(s,\chi)| = |\beta(1-s,\overline{\chi})| \) can happen at most once in the region \( 0 < \re s < 1 \) when the imaginary part is fixed. Moreover we know that equality holds at \( \re s = 1/2 \), because for such a \( s \),
\begin{align*}
\beta(1-s,\overline{\chi}) & = \beta(\overline{s},\overline{\chi}) \\
& = \overline{\beta(s,\chi)}.
\end{align*}
This shows the claimed equivalence.
\end{proof}

We now specialize the previous theorem when \( \chi \) is real, odd and for real zeros of \( L \). Recall that for such a \( \chi \), Gauss himself proved that \( w_\chi = 1 \). Therefore, as in the previous proof, for real \( 0 < s < 1 \), if \( \xi(s,\chi) \neq 0 \), then 
\[ \lim_{n \to \infty}\frac{\xi_n(s,\chi)}{\xi_n(1-s,\overline{\chi})} = 1. \]
Moreover, for \( 0 < s < 1 \), \( \beta(s,\chi) \) is real and strictly decreasing. Therefore, in the case \( \xi(s,\chi) = 0 \), then we obtain the equivalence 
\begin{align*}
    \lim_{n \to \infty}\frac{\xi_n(s,\chi)}{\xi_n(1-s,\overline{\chi})} = 1 & \iff \beta(s,\chi) = \beta(1-s,\overline{\chi}) \\
    & \iff s = 1/2.
\end{align*}
We therefore have proven Corollary \ref{GRH in the real line for odd characters}.
\begin{remark}
Using the exact same method as in the proof of Theorem \ref{GRH for odd characters}, one can reduce the lower bound of \( 8 \) in the imaginary part of Theorem \ref{GRH equivalence for even} to less than \( 1/2 \). However, we believe that one would need an improvement of this monotonicity argument to erase completely this lower bound.
\end{remark}

\section{Real zeros of Dirichlet $L$-functions}
\label{sec:realseros}

\subsection{No real zero for Dirichlet $L$-functions with heat positive character}
Let $\chi$  be an even, real and primitive Dirichlet character modulo $q>1$. The starting point is the following asymptotics first established by Friedli and reproved in section \ref{Section for asymptotic for odd character} above, Theorem \ref{Asymptotic for even}:

For any fixed \( s \in \CC \), the following asymptotics, for \( n \to \infty \), hold
\begin{equation}\label{Asymptotic of size 2 for even characters L-functions}
    L_n(s,\chi) = 2\left(\frac{qn}{2\pi}\right)^{2s}\left(L(s,\chi)+\frac{s}{12}\left(\frac{qn}{2\pi}\right)^{-2}L(s-2,\chi)+\mathcal{O}(n^{-4})\right). 
\end{equation}

Now note that since \( L(s,\chi) > 0 \) for \( s > 1 \) real, the functional equation for even and primitive characters tells us that \( L(s,\chi) \) is negative for \( -2 < s < 0 \). Indeed, for \( \chi \) primitive of conductor \( q > 1 \), real and even, we know that \( L(s, \chi) \) is entire and has the functional equation
\begin{equation}
\xi(s,\chi) = \frac{\tau(\chi)}{\sqrt{q}}\xi(1-s,\chi),
\end{equation}
where \( \xi(s,\chi) \ldef (\pi/q)^{-s/2}\Gamma(s/2)L(s,\chi) \) and \( \tau(\chi) \) is the Gauss sum. It is well-known that for the characters we consider we have \( \tau(\chi) = \sqrt{q} \), which shows that \( \xi(s,\chi) \) is even around \( s = 1/2 \). 

We say that \( \chi \) is \emph{heat positive} if for some sequence $n_i\rightarrow \infty$, the following inequality, 
\begin{equation}
\sum_{j=1}^{qn_i-1}\chi(j)e^{-t4\sin^2(\pi j/qn_i) }\geq 0   \label{eq:heatpositive}
\end{equation}
holds for all $t\geq 0$

We have the following corollary of the asymptotics:

\begin{corollary} \label{cor:real zeros}
Let \( \chi \) be a real, even, primitive Dirichlet character of conductor \( q > 1 \) and which is heat positive. Then $L(s,\chi)>0$  for \( 0<s<1 \).
\end{corollary}

\begin{proof}
The Mellin transform of the sum in the hypothesis gives $L_n (s,\chi)\cdot \Gamma(s)$. The \( \Gamma \) factor is positive on $s>0$, and hence heat positivity implies that $L_{n_i} (s,\chi)>0$. Therefore, since $L(s-2,\chi)<0$ for $0<s<1$ as remarked above, the asymptotics \eqref{Asymptotic of size 2 for even characters L-functions} force $L(s,\chi)>0 $.    
\end{proof}

Which characters are heat positive? Characters with small $q$ and with many partial sums of the character being positive. This points to Chowla and is closely linked to the classical theta function approach, see \cite{LM13}. For this reason, at present time, this section does not give any new results on the positivity of $L(s,\chi)$ on $0<s<1$. The question is whether the Mellin transform is always positive even when there is a window of $t$'s, moving with $n$, where the character is not heat positive.

\subsection{First examples of heat positive characters}

For illustration, we will prove a simple result: 
\begin{proposition}\label{thm:heatpositive}
    Let \( \chi \) be a real, even, primitive Dirichlet character, primitive of conductor \( q > 1 \). Then the inequality (\ref{eq:heatpositive}) holds for all $t$ and $n$ such that  $0<t < n^2 / 2\log 2$ and  $t>q^{2}n^{2}\log2/4(16-\pi^{2})$.
\end{proposition}

Due to the appearance of $q$ in one inequality and not the other, the proposition only implies that all characters with $q<7$ heat positive (for all $n$). This can be significantly strengthened assuming some partial sums of $\chi$ positive and using partial summation. One such case is treated in the next subsection.

The proof of Proposition \ref{thm:heatpositive} starts by writing \( L_n(s,\chi) \) as the Mellin transform of a solution of a particular discrete heat equation. Using Fourier transform and the heat kernel on the discrete circle \( \ZZ/qn\ZZ \), one gets two expressions for the heat solution.

Recall that the discrete Laplace operator \( \Delta \) on the discrete group \( \ZZ/qn\ZZ \) is defined by 
\begin{equation}
(\Delta f)(x) \ldef 2f(x)-f(x+1)-f(x-1),
\end{equation}
acting on functions \( f \in L^2(\ZZ/qn\ZZ) \). This operator is positive semi-definite and the eigenvalues are
\begin{equation}
\lambda_j \ldef 4 \sin^2\left(\frac{\pi j}{qn}\right)\text{, with } j \in \ZZ/qn\ZZ.
\end{equation} 
The discrete heat kernel associated is defined to be the one parameter family of operator 
\[ t \mapsto e^{-t\Delta}, \]
and is realized by convolution against the function
\begin{equation}
K_n(t,x) \ldef \frac{1}{qn}\sum_{j = 0}^{qn-1}e^{-t4\sin^2\left(\pi j /qn\right)}e^{2 \pi i xj / qn}.
\end{equation}
This function is the fundamental solution of the discrete heat equation 
\begin{equation}
\begin{cases}
(\Delta+\partial_t)K_n = 0 \\
K_n\big|_{t = 0} = \delta_0.
\end{cases}
\end{equation}
Let \( \delta_j \) be the Dirac function at \( j \in \ZZ/qn\ZZ \), that is \( \delta_j(x) = 1 \) if \( x = j \) and \( \delta_j(x) = 0 \) else.
Let us now consider the discrete heat equation with a special initial condition associated to the character \( \chi \), namely
\begin{equation}\label{Twisted discrete heat equation}
\begin{cases}
(\Delta + \partial_t)u = 0 \\
u\big|_{t = 0} = \chi^\flat,
\end{cases}
\end{equation}
with the initial condition being \( \chi^\flat \ldef \sum_{j = 1}^{q-1}\chi(j)\delta_{jn} \). 

The heat kernel on $\mathbb{Z}$ is the discrete Gaussian $K_{\ZZ}(t,n) = e^{-2t}I_n(2t)$, where $I$ denotes the $I$-Bessel function, see \cite{CJKS25} for a discussion on this. From usual consideration on covering of graphs, one obtains the formula 
\begin{equation}\label{Heat Kernel for covering Z to Z/qnZ}
    K_n(t,x) = \sum_{j \in \ZZ}K_{\ZZ}(t,x+jqn) = e^{-2t}\sum_{j \in \ZZ}I_{x+jqn}(2t)
\end{equation}

\begin{proposition}
The solution to \eqref{Twisted discrete heat equation} has the following two expressions
\begin{equation*}
\frac{\tau(\chi)}{qn}\sum_{\omega = 1}^{qn-1}\chi(\omega)e^{-t4\sin^2\left(\frac{\pi \omega}{qn}\right)}e^{2 \pi i \frac{x \omega}{qn}} = \sum_{j = 1}^{q-1}\chi(j) K_n(t,x-jn).
\end{equation*}
In particular, for $x=0$
\begin{equation}\label{Two expressions of the trace}
\tau(\chi) \sum_{j = 1}^{qn-1}\chi(j)e^{-t4\sin^2\left(\frac{\pi j}{qn}\right)} = qn \sum_{j = 1}^{q-1}\chi(j) K_n(t,jn).
\end{equation}
\end{proposition}
\begin{proof}
First, taking the Fourier transform of the equation \eqref{Twisted discrete heat equation} leads to the equation 
\begin{equation}
\begin{cases}
(\Psi_n + \partial_t)\widehat{u} = 0 \\
\widehat{u}\big|_{t = 0} = \chi^\sharp,
\end{cases}
\end{equation}
where \( \chi^\sharp = \widehat{\chi^\flat} \) is the Fourier transform of \( \chi^\flat \) and where \( \Psi_n \) is the diagonal operator of multiplication by the spectrum of \( \Delta \), namely
\[ (\Psi_n \rho)(\omega) = 4\sin^2\left(\frac{\pi \omega}{qn}\right) \rho(\omega). \]
One computes for \( \omega \in \ZZ/qn\ZZ \),
\begin{align*}
\chi^\sharp(\omega) & = \sum_{j = 0}^{q-1} \chi(j) \widehat{\delta_{jn}}(\omega) \\
& = \sum_{j = 0}^{q-1}\chi(j)e^{-2 \pi i j\omega/q}.
\end{align*}
But it is well known that when \( \chi \) is primitive, this last sum is \( \overline{\chi(-\omega)} \tau(\chi) \) where \( \tau(\chi) \) is the Gauss sum. Therefore, since \( \chi \) is supposed to be real, even and primitive, one gets 
\[ \widehat{u}(t,\omega) = \tau(\chi)\chi(\omega)e^{-t4\sin^2\left(\frac{\pi \omega}{qn}\right)}. \]
Taking the inverse Fourier transform and discarding the term \( \chi(0) = 0 \) it gives
\begin{equation}
u(t,x) = \frac{\tau(\chi)}{qn}\sum_{\omega = 1}^{qn-1}\chi(\omega)e^{-t4\sin^2\left(\frac{\pi \omega}{qn}\right)}e^{2 \pi i \frac{x \omega}{qn}},
\end{equation}
since we consider \( \ZZ/qn\ZZ \) as discrete, the dual measure (with respect to Fourier) is therefore the compact one on \( \ZZ/qn\ZZ \). Taking the trace of this solution viewed as an operator acting on \( L^2(\ZZ/qn\ZZ) \) by convolution amounts to putting \( x = 0 \) and multiplying by \( qn \), which gives 
\begin{equation}
\Tr (u(t)) = \tau(\chi) \sum_{j = 1}^{qn-1}\chi(j)e^{-t4\sin^2\left(\frac{\pi j}{qn}\right)}.
\end{equation}
Now, as usual, the solution \( u \) is as well given by convolution of the heat kernel with the initial condition \( \chi^\flat \),
\begin{align*}
u(t,x) & = \left(K_n(t) \ast \chi^\flat\right)(x) \\
& = \sum_{j = 0}^{q-1}\chi(j) \cdot \left(K_n(t) \ast \delta_{jn}\right)(x) \\
& = \sum_{j = 0}^{q-1}\chi(j) K_n(t,x-jn).
\end{align*}
The symmetry of the heat kernel \( K_n(t,-x) = K_n(t,x) \) finishes the proof.
\end{proof}

Taking the Mellin transform of the left hand side of \eqref{Two expressions of the trace} gives for \( \re s > 0 \)
\begin{equation*}
\frac{1}{\Gamma(s)}\int_0^{\infty}\Tr(u(t))t^s \frac{\mathrm{d}t}{t} = \tau(\chi)L_n(s,\chi).
\end{equation*}

Note that the formula \eqref{Heat Kernel for covering Z to Z/qnZ} implies
\begin{align*}
\sum_{j=1}^{q-1}\chi(j)K_{n}(t,jn)& =e^{-2t}\sum_{l=-\infty}^{\infty}\sum_{j=1}^{q-1}\chi(j)I_{n(j+lq)}(2t) \\
& =e^{-2t}\sum_{m=-\infty}^{\infty}\chi(m)I_{mn}(2t) \\
&=2e^{-2t}\sum_{m=1}^{\infty}\chi(m)I_{mn}(2t).
\end{align*}

The last inequality is true since $\chi$ is even, $I_{-k}=I_{k}$, and since $\chi(0)=0$.
The idea is now that for any fixed $t>0$, as $n\rightarrow\infty$,
the whole expression tends to $0$, but with the $m=1$ term dominating, giving something positive, since $\chi(1)=1$ and the heat kernel is positive. 

Since $\chi$ takes values in $-1,0,1$ and the $I$-Bessel function
is positive, we have in the worst case:
\[
\sum_{j=1}^{q-1}\chi(j)K_{n}(t,jn)>2e^{-2t}\left(I_{n}(2t)-\sum_{m=2}^{\infty}I_{mn}(2t)\right).
\]

We proceed by comparing $I_{n}$ with $I_{mn}$. In \cite{Am74} one finds that
\[
\frac{I_{k+1}(x)}{I_{k}(x)}\leq\frac{x}{k+\sqrt{k^{2}+x^{2}}}\leq\frac{x}{k+x}
\]
for $k\leq0$ and $x>0$. This leads to
\[
\frac{I_{(m+1)n}}{I_{mn}}\leq\frac{I_{mn+1}}{I_{mn}}\frac{I_{mn+2}}{I_{mn+1}}...\frac{I_{(m+1)n}}{I_{mn+n-1}}\leq\frac{x}{mn+x}\frac{x}{mn+1+x}...\frac{x}{mn+n-1+x}
\]
\[
\leq\left(\frac{x}{n+x}\right)^{n}.
\]
Let $c_{n}(t)=\left(t/(n+t)\right)^{n}<1$. This means that
$$
I_n(t)\geq I_{mn}(t)c_n(t)^{m-1}.
$$
We apply this inequality to our expression:
$$
\sum_{m=1}^{\infty}\chi(m)I_{mn}(2t)>I_n(2t)\left( 1 - \sum_{m=2}^{\infty} c_n(2t)^{m-1}\right).
$$
So the question is now when do we have 
$$
\sum_{m=2}^{\infty} c(2t)^{m-1}=c_n(2t)/(1-c_n(2t))<1.
$$
In other words $c_n(2t)<1/2$. This is equivalent to
\[
(n+2t)^{n}>2^{1+n}t^{n}
\]
and
\[
n+2t>2^{1/n+1}t
\]

\[
n>2t(2^{1/n}-1)=2t(1+\frac{1}{n}\log2+\frac{1}{2}\left(\frac{\log2}{n}\right)^{2}+...-1)>2t\frac{\log2}{n}
\]
In conclusion, whenever $n^2 > 2\log 2 \cdot t$ the heat solution at $x=0$, the right hand side of (\ref{Two expressions of the trace}) is positive. Numerically, for all $0<t<0.72n^2 $ the expression is positive.

We now study large $t$ using instead the spectral expression of the heat solution. In other words, let us consider the other expression of the identity (\ref{Two expressions of the trace}), taking away $\tau(\chi)$.
So we are interested in 
\[
\sum_{j=1}^{qn-1}\chi(j)e^{-t4\sin^{2}(\pi j/qn)}
\]
for all large $t$ and fixed $n$. The idea is again that $j=1$
and $j=qn-1$ should dominate. Using that $\chi$ is even, we have
\[
\sum_{j=1}^{(qn-1)/2}\chi(j)e^{-t4\sin^{2}(\pi j/qn)}>\chi(1)e^{-4t\sin^{2}(\pi/qn)}-\sum_{j=2}^{(qn-1)/2}e^{-t4\sin^{2}(\pi j/qn)}=
\]
\[
=e^{-4t\sin^{2}(\pi/qn)}\left(1-e^{4t\sin^{2}(\pi/qn)}\sum_{j=2}^{(qn-1)/2}e^{-t4\sin^{2}(\pi j/qn)}\right).
\]
Focusing on the sign of the parentheses, we now use the inequalities
$x\frac{2}{\pi}\leq\sin(x)\leq x$ in $0\leq x\leq\pi/2.$ 
\[
1-e^{4t\sin^{2}(\pi/qn)}\sum_{j=2}^{(qn-1)/2}e^{-t4\sin^{2}(\pi j/qn)}>1-e^{4t\pi^{2}/q^{2}n^{2}}\sum_{j=2}^{(qn-1)/2}e^{-4t\pi^{2}j^{2}/q^{2}n^{2}\cdot\left(4/\pi^{2}\right)}>
\]
\[
>1-e^{4t\pi^{2}/q^{2}n^{2}}e^{-16t4/q^{2}n^{2}}\frac{1}{1-e^{-16t5/q^{2}n^{2}}}.
\]

So we want t such that
\[
e^{4t\pi^{2}/q^{2}n^{2}}e^{-16t4/q^{2}n^{2}}<1-e^{-16t5/q^{2}n^{2}}
\]
or equivalently
\[
e^{4t\pi^{2}/q^{2}n^{2}}e^{-16t4/q^{2}n^{2}}+e^{-16t5/q^{2}n^{2}}<1.
\]
It is valid if $t>q^{2}n^{2}\log2/4(16-\pi^{2})\approx0.028q^{2}n^{2}$
or with better numeric $t>0.015q^{2}n^{2}$. Compare this to the above $0.72.$
So we covered all of t as long as $q<7.$ 

In conclusion, this shows that all even, real, primitive character with conductor less than or equal to $7$ are heat positive, and their Dirichlet $L$-function has no real zeros in the unit interval.

\subsection{Further examples of heat positive characters}

With another argument, we show heat positivity without any a priori condition on $q$, but under a stronger hypothesis on the character. 
\begin{definition}[The positive mean property]
Let \( \chi \) be a Dirichlet character, primitive of conductor \( q > 1 \), real and even. We say that \( \chi \) have the \emph{positive mean property} if the following
\[ \sum_{j = 0}^{m}\chi(j) \ge 0, \]
holds for all integer \( 0 \le m \le q/2 \).
\end{definition}
Note that in the literature, such positivity is usually assumed to hold for all $m<q$ in the case of odd characters.

\begin{theorem}\label{thm:positivemean}
    Let \( \chi \) be a Dirichlet character, primitive of conductor \( q > 1 \), real, even and with the positive mean property. Then $\chi $ is heat positive and $L(s,\chi)>0$ for $0<s<1$.
\end{theorem}

This was already proved by Chowla in \cite{Ch36} in a quick way with a clever formula (corresponding to the case that $m(\chi)=2$).

\begin{proof}

Look at the right hand side of \eqref{Two expressions of the trace}, let \( q' \ldef \lfloor q/2 \rfloor \) and use symmetry plus Abel summation to get
\begin{align*}
\sum_{j = 0}^{q-1}\chi(j)K_n(t,jn) & = 2\sum_{j = 0}^{q'}\chi(j)K_n(t,jn) \\
& = 2 \left[\sum_{j = 0}^{q'}\chi(j)\right]K_n(t,q'n) + \\
& - 2 \sum_{j = 0}^{q'-1}\left[\sum_{k = 0}^{j}\chi(k)\right](K_n(t,(j+1)n)-K_n(t,jn)).
\end{align*}
The expressions in bracket \( [.] \) are positive by assumption on \( \chi \), it remains only to show that the heat kernel is positive and decreasing with respect to the distance from \( 0 \) in the Cayley graph of \( \ZZ/qn\ZZ \) with the generating set \( \{ \pm 1 \} \). 

Precisely, let \( N \ge 3 \), and \( \Delta \) be the discrete Laplace operator on \( \ZZ/N\ZZ \). Let also \( K(t) = \exp(-t\Delta)\delta_0 \) the discrete heat kernel on \( \ZZ/N\ZZ \) and define for \( x \in \ZZ/N\ZZ \), \( h(x) \ldef \min\{|x|,|N-x|\} \) to be the distance of \( x \) from \( 0 \) in the Cayley graph of \( \ZZ/N\ZZ \).
The required monotonicity is established in Lemma \ref{lem:decreasing} below.

    Therefore we have proven that the trace \( \Tr (u(t)) \) is positive for all \( t \ge 0 \), but since \( \Gamma(s) > 0 \) for all \( s > 0 \), we conclude that the Mellin transform of \( \Tr (u(t)) \) is also positive in the range \( s > 0 \). This finishes the proof of Theorem \ref{thm:positivemean}, in view of Corollary \ref{cor:real zeros}. The positivity of $\tau(\chi)$ follows from combining this with the positivity established in Theorem \ref{thm:heatpositive}.
\end{proof}

\begin{lemma}\label{lem:decreasing}
With the notations as above, the heat kernel \( K(t) \) is positive and decreasing with respect to \( h \). That is for all \( x,y \in \ZZ/N\ZZ \), we have 
\[ h(x) \le h(y) \Longrightarrow K(t,y) \le K(t,x). \]
\end{lemma}

\begin{proof}
Consider the operator \( S \ldef 3-\Delta \), defined by \( (Sf)(x) \ldef f(x-1)+f(x)+f(x+1) \), which is the generator of the random walk on \( \ZZ/N\ZZ \) allowing the walk to stand still. Let us also consider \( \partial \), the forward difference operator, defined by \( (\partial f)(x) \ldef f(x+1)-f(x) \). Then \( K(t) = e^{-3t}\exp(tS)\delta_0 \), and therefore we need only to analyze \( \phi(t) \ldef \exp(tS)\delta_0 \). But clearly, 
\[ \phi(t,x) = \sum_{n \ge 0}\frac{t^n}{n!}c_n(x), \]
where 
\[ c_n(x) = \#\{\text{paths from } 0 \text{ to } x \text{ of length } n \text{, allowing the path to stand still} \}, \]
which shows that \( \phi(t) \) is positive, but most importantly, the function \( c_n \) satisfy the recurrence relation 
\[ c_{n+1} = Sc_n \]
for all \( n \ge 0 \). Therefore, for all \( n \ge 0 \), we have 
\[ \partial c_{n+1} = S\partial c_n. \]
Now, for \( n = 0 \), we have \( c_0 = \delta_0 \) which indeed is decreasing with respect to the height \( h \). By induction, suppose that \( c_n \) is decreasing with respect to the height \( h \). It means exactly that \( \partial c_n (x) \le 0 \) for all \( 0 \le x < N' \), where \( N' \ldef \lfloor N/2 \rfloor \), now look at 
\begin{align*}
c_{n+1}(x+1)-c_{n+1}(x) & = (c_n(x)-c_n(x-1)) \\
& + (c_n(x+1)-c_n(x)) \\
& + (c_n(x+2)-c_n(x+1)).
\end{align*}
If \( 0<x<N'-1 \), then every term on the right is negative by hypothesis on \( c_n \), if \( x = 0 \), then the right side is reduce to \( c_n(2)-c_n(1) \) which is negative and if \( x = N'-1 \), then the right side is reduce to \( c_n(N'-1)-c_n(N'-2) \) which is also negative. We have proven that \( \partial c_{n+1} (x) \le 0 \) for all \( 0 \le x < N' \), which conclude the induction and therefore the proof that \( \phi(t) \) is decreasing with respect to \( h \), so is \( K(t) \).
\end{proof}

\section{Special values}\label{sec:special}
In this section, we prove Theorem \ref{Theorem:trig}, and by doing so we link special values of both discrete zeta functions and discrete \( L \)-functions with special values of the Riemann zeta function and Dirichlet \( L \)-functions, respectively. For this, we first use the Laplace transform of the discrete heat kernel on the \( \theta \)-bundle over \( \ZZ/n\ZZ \) together with a precise expression of the characteristic polynomial of the bundle Laplacian \( \Delta_{n,\theta} \) (namely, the Chebyshev polynomials) to obtain a polynomial behavior of the discrete special value. Combining this polynomial behavior with the asymptotics of Theorem \ref{Asymptotic for the bundle Laplacian}, we obtain closed formulas for the zeta functions associated. By using standard formulas for \( L \)-functions and specializing \( \theta \) accordingly, we obtain the desired relation between discrete and continuous zeta and \( L \)-functions.

\subsection{Formulas}
In terms of discrete zeta function on the \( \theta \)-bundle, Theorem \ref{Theorem:trig} says that for any \( 0 < \theta < 1 \) and integers \( n,m \ge 1 \), we have
\begin{equation}\label{Even part of Theorem 1}
\zeta_{\ZZ/n\ZZ}(m,\theta) = \sum_{k = 0}^m a_{m-k}(m)\zeta_{\RR/\ZZ}(k,\theta)n^{2k}.
\end{equation}

As above, the coefficients \( m \mapsto a_k(m) \) are polynomial functions of degree \( k \), given by the Taylor expansion 
\[ \left(\frac{z/2}{\sin z/2}\right)^{2s} = \sum_{k \ge 0} a_k(s) z^{2k}. \]

The Theorem \ref{Theorem:trig} leads to the formulas in Corollary \ref{Corollaryspecial}, linking discrete and classical special values for zeta and Dirichlet \( L \)-functions. They were first obtained by Xie, Zhao, and Zhao in the case \( n = 1 \).  

Moreover, one can deduced the case of \( \theta = 0 \) as well, which has a long history and is connected to Verlinde formulas, see for example \cite[Theorem 1.iii]{Zag96}. Here we put large emphasis on spectral zeta function involved in this formula. Namely,

\begin{corollary}\label{Corollary no character}
For any integers \( n,m \ge 1 \), the following holds:
\[ \zeta_{\ZZ/n\ZZ}(m) = \sum_{k = 0}^{m}a_{m-k}(m)\zeta_{\RR/\ZZ}(k)n^{2k}. \]
\end{corollary}

\subsection{Characteristic polynomial of the bundle Laplacian.}
Recall from \ref{subsec:DiscreteBundle} that \( \Delta_{n,\theta} \) is the bundle Laplacian on \( \ZZ/n\ZZ \) associated to the parameter \( \theta \in \RR \) defined by
\[ (\Delta_{n,\theta}f)(x) \ldef 2f(x)-e^{2\pi i \theta/n}f(x+1)-e^{-2\pi i \theta/n}f(x-1), \]
on functions \( f \colon \ZZ/n\ZZ \rightarrow \CC \).

The characteristic polynomial of \( \Delta_{n,\theta} \) is 
\[ P_{n,\theta}(x) = \det(\Delta_{n,\theta}-x) = \prod_{j = 0}^{n-1}\left[4\sin^2\left(\frac{\pi (j+\theta)}{n}\right)-x\right], \]
and in fact it can be expressed in terms of Chebyshev polynomials.
\begin{lemma}\label{Lemma Chebychev and characteristic polynomial}
Let \( T_n \) be the \( n \)-th Chebyshev polynomial of the first kind. Then the characteristic polynomial of \( \Delta_{n,\theta} \) satisfies 
\begin{equation}
T_n(1-2x)-\cos(2 \pi \theta) = \frac
12 P_{n,\theta}(4x).
\end{equation}
\end{lemma}
\begin{proof}
Recall that \( T_n \) satisfies 
\[ T_n(\cos(x)) = \cos(nx), \]
and therefore the polynomial 
\[ Q_{n,\theta}(x) \ldef T_n(1-2x)-\cos(2 \pi \theta) \]
vanishes at the numbers 
\[ \mu_j \ldef \sin^2\left(\frac{\pi(j+\theta)}{n}\right). \]
If \( \theta \notin \frac{1}{2}\ZZ \), then all the \( \mu_j \) are distinct and therefore 
\[ Q_{n,\theta}(x) = \frac{1}{2}\prod_{j=0}^{n-1}\left[4\sin^2\left(\frac{\pi(j+\theta)}{n}\right)-4x\right], \]
since \( T_n(x) \) is of degree \( n \) and has leading coefficient \( 2^{n-1} \). By continuity of the map \( \theta \mapsto Q_{n,\theta}(x) \), this last factorization of \( Q_{n,\theta} \) holds for any \( \theta \in \RR/\ZZ \). This implies the desired result.
\end{proof}
\begin{remark}
The expansion of the Chebyshev polynomial \(T_n\) near \(z=1\) (see p.~779 of \cite{AS64}), combined with Lemma~\ref{Lemma Chebychev and characteristic polynomial}, gives
\begin{equation}\label{Expansion at z = 1 of T_n}
\det(x+\Delta_{n,\theta})
=
2(1-\cos(2\pi \theta))
+
2n \sum_{k = 1}^{n}
\frac{(n+k-1)!}{(n-k)!(2k)!}x^k .
\end{equation}

This formula admits a natural combinatorial interpretation. When \(\theta=0\), we recover the coefficients of the characteristic polynomial of the usual combinatorial Laplacian \(\Delta_n\) on the cycle graph \(\ZZ/n\ZZ\). By the rooted-forest version of the matrix-tree theorem, see \cite{KC74,CL96}, the coefficient of \(x^k\) in \(\det(x+\Delta_n)\) is the number of rooted spanning forests of \(\ZZ/n\ZZ\) with \(k\) connected components.

There is also a version of the matrix-tree theorem for bundle graphs and therefore in the case of a general \( \theta \), see \cite{Fo93,Ke11,BPT15}. In this case, spanning forests are replaced by mixed spanning forests, namely spanning subgraphs whose connected components are either rooted trees or unicyclic components (not rooted). The unicyclic components are counted with respect to their monodromy. We refer to \cite[Theorem~3.2]{BPT15} for the precise statement.

It is clear, however, that in the particular case of the cycle graph \(\ZZ/n\ZZ\), there is only one cycle, namely the whole graph. Its monodromy is \(e^{2\pi i\theta}\), and therefore the unique cyclic contribution is given by
\[
(1-e^{2\pi i\theta})(1-e^{-2\pi i\theta})
=
2(1-\cos(2\pi\theta)).
\]
This is exactly the constant term in \eqref{Expansion at z = 1 of T_n}. All the higher coefficients come from rooted spanning forests of \( \ZZ/n\ZZ \), and hence are independent of \( \theta \).
\end{remark}
 
From this combinatorial perspective, the special values \( \zeta_{\ZZ/n\ZZ}(k,\theta) \) for \( k \in \ZZ \) can be expressed in terms of the number of rooted spanning forests on \( \ZZ/n\ZZ \), since they are symmetric expressions of the eigenvalues of \( \Delta_{n,\theta} \). Combining this observation with the asymptotics of \Cref{Asymptotic in the case d=1} lead to evaluations of the special values \( \zeta_{\RR/\ZZ}(k,\theta) \). Let us work through the special case of \( \zeta(2) \).

\begin{example}
When \( \theta = 0 \) write 
\[ \det(x+\Delta_n) = \sum_{j = 1}^{n}a_j(n)x^j, \]
and therefore, for \( s = 1 \),
\[ \zeta_{\ZZ/n\ZZ}(1) = \sum_{j = 1}^{n-1}\frac{1}{4\sin^2\left(\tfrac{\pi j}{n}\right)} = \frac{a_2(n)}{a_1(n)}. \]
But looking at the asymptotics of \Cref{Asymptotic in the case d=1} in the special case \( s = 1 \) and keeping in mind the trivial zeros of both \( \zeta \) and \( \zeta_\ZZ \) we obtain 
\[ \frac{a_2(n)}{a_1(n)} = n^2 \cdot \frac{2}{(2\pi)^2}\zeta(2) + \frac{2}{12}\zeta(0) + \mathcal{O}(n^{-m}),  \]
for any \( m \ge 0 \). Therefore, using the polynomiality of \( a_j \) we conclude that the asymptotics must stop and in fact is an exact formula, that is, for all \( n \ge 1 \)
\[ \frac{a_2(n)}{a_1(n)} = \zeta_{\ZZ/n\ZZ}(1) = n^2 \cdot \frac{2}{(2\pi)^2}\zeta(2) + \frac{2}{12}\zeta(0). \]
Now, take \( n = 2 \) and count the rooted spanning trees and the rooted spanning $2$-forests of \( \ZZ/2\ZZ \) as in Fig.~\ref{fig:spanning} to obtain 
\[ \frac{1}{4} = 2^2 \cdot \frac{2}{(2\pi)^2}\zeta(2) + \frac{2}{12}\left(-\frac{1}{2}\right), \]
which implies \( \zeta(2) = \tfrac{\pi^2}{6} \). A similar calculation holds for the values \( \zeta(2k) \), \( k \ge 0 \), indeed for \( \zeta_{\RR/\ZZ}(k,\theta) \) as well. In particular, the value of \( \zeta(0) \) can also be deduced in this way, even directly by looking at the asymptotics of \Cref{Asymptotic in the case d=1} for \( s = 0 \) : 
\[ n-1 = n\zeta_\ZZ(0) + 2\zeta(0) + \mathcal{O}(n^{-2}). \]
\end{example}
\begin{remark}
The polynomial nature of the \( a_j \) is an important feature of the previous calculation, since it is what allows us to assert that the asymptotic expansion of \Cref{Asymptotic in the case d=1} stops when \( s = 1 \). For the purpose of generalizing this phenomenon to \( s \in \ZZ \), let us define
\[
\mathcal P_{\theta}(n,x)
\coloneqq
\det(x+\Delta_{n,\theta}),
\]
and observe that \eqref{Expansion at z = 1 of T_n} implies
\[
\mathcal P_\theta \in (\CC[n])[[x]],
\]
which is invertible in this ring if and only if its constant term is nonzero. That is, if and only if
\[
2(1-\cos(2\pi\theta))\neq 0,
\]
or equivalently \(\theta\notin\ZZ\). This observation will be central to show that the asymptotics of \Cref{Asymptotic in the case d=1} stops when \( s \in \ZZ \).
\end{remark}
\subsection{Laplace transform of the bundle heat kernel}
The heat kernel of the bundle Laplacian \( \Delta_{n,\theta} \) on \( \ZZ/n\ZZ \) is defined to be the one-parameter family of operators
\[ K_{n,\theta}(t) \ldef e^{-t\Delta_{n,\theta}}, \]
with trace  
\[ \Tr[K_{n,\theta}(t)] = \sum_{j=0}^{n-1}e^{-t4\sin^2\left(\frac{\pi (j+\theta)}{n}\right)}. \]
Therefore its Laplace transform is given for \( \re s < 0 \) by
\begin{align*}
\mathcal{L}[\Tr[K_{n,\theta}]](s) & = \int_0^{\infty}e^{st}\Tr[K_{n,\theta}(t)]\mathrm{d}t \\
& = \sum_{j = 0}^{n-1} \int_0^{+\infty}e^{-t(4\sin^2\left(\frac{\pi (j+\theta)}{n}\right)-s)}\mathrm{d}t \\
& = \sum_{j = 0}^{n-1} \frac{1}{4\sin^2\left(\frac{\pi (j+\theta)}{n}\right)-s} \\
& = -\frac{\mathrm{d}}{\mathrm{d}s}\log [P_{n,\theta}(s)] \\
& = \frac{1}{2}\frac{T_n'(1-s/2)}{T_n(1-s/2)-\cos(2 \pi \theta)} \\
& = \frac{1}{2}\frac{nU_{n-1}(1-s/2)}{T_n(1-s/2)-\cos(2 \pi \theta)}.
\end{align*}
where we used that \( T_n' = nU_{n-1} \), with \( U_n \) the \( n \)-th Chebyshev polynomial of the second kind. If \( \theta \notin \ZZ \), then 
\[ \sum_{j=0}^{n-1}\frac{1}{4\sin^2\left(\frac{\pi (j+\theta)}{n}\right)-s} = \sum_{m \ge 0}s^m\sum_{j=0}^{n-1}\left[4\sin^2\left(\frac{\pi (j+\theta)}{n}\right)\right]^{-m-1}. \]
In fact, we have just shown:
\begin{lemma}\label{lemma 2}
If \( 0 < \theta < 1 \), the following equality holds:
\begin{equation}
\mathcal{L}[\Tr[K_{n,\theta}]](s) = \sum_{m \ge 0}s^m\zeta_n(m+1,\theta) = \frac{1}{2}\frac{nU_{n-1}(1-s/2)}{T_n(1-s/2)-\cos(2 \pi \theta)},
\end{equation}
in a non-trivial interval \( - \varepsilon_{\theta} <  s < 0 \), for some \( \varepsilon_{\theta} > 0 \).
\end{lemma}
As a result, we obtain the polynomial behavior of the special value of \( \zeta_{\ZZ/n\ZZ}(m,\theta) \).
\begin{corollary}\label{Polynomial behaviour}
For any \( 0 < \theta < 1 \) and for any integer \( m \in \ZZ \), the map
\[ n \mapsto \zeta_n(m,\theta) \]
is a polynomial function on the domain \( \{ n \in \NN \mid n > -m \} \).
\end{corollary}
\begin{proof}
The case of positive \( m \ge 1 \) is a direct consequence of Lemma \ref{lemma 2}. Indeed, for each \( n \ge 1 \) the generating function of the numbers \( \zeta_n(m,\theta) \), \( m \ge 1 \) converges for some \( s \neq 0 \) and satisfies 
\[ \sum_{m \ge 1}s^m\zeta_n(m,\theta) = s\cdot \frac{nU_{n-1}(1-s/2)}{2(T_n(1-s/2)-\cos(2 \pi \theta))}. \]
By the equation \eqref{Expansion at z = 1 of T_n}, both the numerator and the denominator of the right hand side are in \( (\CC[n])[[s]] \), with the denominator invertible in this ring since \( T_n(1) = 1 \).

For non-positive \( -m \le 0 \), let \( n > m \) and expand the sine function into exponential 
\begin{align*}
    \zeta_n(m,\theta) & = 4^m\sum_{j = 0}^{n-1}\left[\sin\left(\frac{\pi(j+\theta)}{n}\right)\right]^{2m} \\
    & = \sum_{k=-r}^r\binom{2m}{m+k}(-1)^ke^{2\pi\theta k/n}\sum_{j=0}^{n-1}e^{2\pi jk/n} \\ 
    & = n\binom{2m}{m}.
\end{align*}
\end{proof}
Observe that in the last calculation, we have explicitly shown that for \( -m \le 0 \),
\begin{equation*}
    \zeta_{\ZZ/n\ZZ}(-m,\theta) = n\zeta_{\ZZ}(-m),
\end{equation*} 
for all \( 0 < \theta < 1 \), and all \( n > m \).

\subsection{Proof of the theorems and their corollaries}

First recall from the asymptotic of Theorem \ref{Asymptotic for the bundle Laplacian} that when \( s \in \CC \) is not a half integer, \( 0 < \theta < 1 \), \( l \ge 0 \), the following holds as \( n \to \infty \);
\begin{equation}\label{Asymptotic Sidi}
\zeta_n(s,\theta)  =  n\zeta_\ZZ(s) + n^{2s}\left(\sum_{k = 0}^{l} a_k(s) \zeta_{\RR/\ZZ}(s-k,\theta)n^{-2k} + \mathcal{O}(n^{-2-2l})\right).
\end{equation} 
Where the \( a_k(s) \) are given in the introduction.

Since \( \zeta_{\RR/\ZZ} \) vanishes at negative integers, one could expect that the asymptotic becomes an exact formula when \( s \) is an integer. In the case of positive \( s = m \ge 1 \) this is the formula \eqref{Even part of Theorem 1}. It will be a direct consequence of Corollary \ref{Polynomial behaviour}, and in fact, as already proven, this phenomenon holds as well for non-positive \( m \le 0 \) and takes the form \eqref{Eq:special value: theta and negative m}.

\begin{proof}[Proof of Theorem \ref{Theorem:trig}]
When \( s = m \) is a positive integer, the left-hand side of the asymptotic of \eqref{Asymptotic Sidi} is a polynomial in \( n \). Since the length \( l \ge 0 \) can be chosen arbitrarily, this implies that the asymptotic expansion terminates and becomes an exact formula. As \( \zeta_\ZZ \) vanishes at positive integers, as \( \zeta_\ZZ \) is the analytical continuation of the binomial coefficients, the formula \eqref{Even part of Theorem 1} follows and therefore Theorem \ref{Theorem:trig}.
\end{proof} 

\begin{proof}[Proof of Corollary \ref{Corollary no character}]
Fix \( n,m \ge 1 \) and let \( 0 < \theta < 1 \). In the left hand side of formula \eqref{Even part of Theorem 1}, the term \( j = 0 \) can be expanded as 
\[ \left[4\sin^2\left(\frac{\pi\theta}{n}\right)\right]^{-m} = \sum_{k = 0}^{m}a_{m-k}(m)\left(\frac{n}{2\pi\theta}\right)^{2k} + \mathcal{O}(\theta^2), \]
move it from the left hand side to the right hand side to obtain :
\begin{align*}
\sum_{j = 1}^{n-1} \left[4\sin^2\left(\frac{\pi (j+\theta)}{n}\right)\right]^{-m} & = \sum_{k = 0}^{m}a_{m-k}(m) \zeta_{\RR/\ZZ}(k,\theta)n^{2k}-\left[4\sin^2\left(\frac{\pi\theta}{n}\right)\right]^{-m} \\
& = \sum_{k = 0}^{m}a_{m-k}(m) (\zeta_{\RR/\ZZ}(2k,\theta)-(2\pi\theta)^{-2k})n^{2k} + \mathcal{O}(\theta^2).
\end{align*}
Now one can take the limit \( \theta \to 0 \) and the desired result follows.
\end{proof}

\begin{proof}[Proof of Corollary \ref{Corollaryspecial}]
Since 
\[ \sum_{r=1}^{q-1}\chi(r)(\partial_\theta^a\zeta_n)(s,r/q)= \left(-s\frac{2\pi}{n}\right)^a \cdot L_n(s,\chi), \]
for \( a \in \{0,1\} \) being the same parity as \( \chi \), the Corollary \ref{Polynomial behaviour} implies that for each \( m \ge 1 \) the map \( n \mapsto L_n(m,\chi) \) is a polynomial and therefore the asymptotic in Theorem \ref{Asymptotic for even} and in Theorem \ref{Asymptotic for L tilde} must stops and becomes exact formulas. For \( m = 0 \) and \( \chi \) even, the corresponding formula of Corollary \ref{Corollaryspecial} is trivial since both side are \( 0 \). For \( m = 0 \) and \( \chi \) odd, one has 
\begin{align*}
    L_n(0,\chi) & = \lim_{s \to 0}\frac{n}{2\pi}\frac{1}{s}\sum_{r=1}^{q-1}\chi(r)(\partial_\theta\zeta_n)(s,r/q) \\
    & = \frac{n}{2\pi}\sum_{r = 1}^{q-1}\chi(r) \frac{\partial}{\partial \theta}\Big|_{\theta = r/q}\zeta_n'(0,\theta).
\end{align*}
But the derivative at \( 0 \) of the spectral zeta function associated to the \( \theta \)-bundle is the log-determinant of \( \Delta_{n,\theta} \). That is 
\[ \zeta_n'(0,\theta) = -\log \det(\Delta_{n,\theta}) = -\log (4\sin^2(\pi\theta)). \]
In particular, it does not depend on \( n \). Thus, \( n \mapsto L_n(0,\chi) \) is a polynomial and the same argument follows. 
One could have argue alternatively that 
\begin{align*}
    L_n(0,\chi) & = \sum_{j=1}^{qn-1}\chi(j)\cot\left(\frac{\pi j}{qn}\right) \\
    & = \sum_{r=1}^{q-1}\chi(r)\sum_{j=0}^{n-1}\cot\left(\frac{\pi r}{qn}+\frac{\pi j}{n}\right) \\
    & = n\sum_{r = 1}^{q-1}\chi(r)\cot\left(\frac{\pi r}{q}\right),
\end{align*}
using a well-known cotangent formula. This results in the same polynomial behavior.
\end{proof}

A.K. Section de mathématiques, Université de Genève, rue du Conseil-Général 7-9, 1205 Genève, Suisse;  Matematiska institutionen, Uppsala universitet, Box 256, 751 05, Uppsala, Sweden,
anders.karlsson@unige.ch

D.M.
Section de mathématiques, Université de Genève, rue du Conseil-Général 7-9, 1205 Genève, Suisse,
dylan.mueller@unige.ch 

\end{document}